\documentclass[a4paper,11pt]{article}

\mathcode`\:="603A     
\mathcode`\<="4268     
\mathcode`\>="5269     
\mathchardef\gt="313E  
\mathchardef\lt="313C  
\mathchardef\colon="303A  

\usepackage{amssymb,amsthm,amsmath}
\usepackage{geometry}
\usepackage{graphicx}
\usepackage[latin1]{inputenc}
\usepackage[small,nohug,heads=vee]{diagrams}
\diagramstyle[labelstyle=\scriptstyle]

\mathcode`\:="603A   
\def\colon{\mathrel:}

\theoremstyle{plain}
\newtheorem{theorem}{Theorem}[section]
\newtheorem{lemma}[theorem]{Lemma}
\newtheorem{proposition}[theorem]{Proposition}

\theoremstyle{definition}
\newtheorem{definition}[theorem]{Definition}

\author{Samuele Maschio}
\title{What is the real category of sets?}
\date{}
\begin{document}
\maketitle
\textbf{Abstract}\\
{\footnotesize According to Kreisel, category theory provides a powerful tool to organize mathematics. A sample of this descriptive power is given by the categorical analysis of the practice of "classes as shorthands" in ZF set theory. In this case category theory provides a natural way to describe the relation between mathematics and metamathematics: if metamathematics can be described by using categories (in particular syntactic categories), then the mathematical level is represented by internal categories. Through this two-level interpretation we can clarify the relation between classes and sets in ZF and, in particular, we can present two equivalent categorical notions of definable set. Some common sayings about set theory will be interpreted in the light of this representation, emphasizing the distinction between naive and rigorous sentences about sets and classes.}
\section{The spiny relation between categories and sets}
The relation between sets and categories is a spiny topic. The set theoretical foundations of category theory were discussed by many mathematicians, e.g.\;by Feferman \cite{FEF}, Engler \cite{ENG} and \cite{LOL}. Viceversa, the categorical foundation of set theory is not a clear matter. Categorical logic allows us to speak about categorical models of mathematical theories and in particular about categorical models of set theories, as for example $IZF$, $CZF$, $ZF$. In 1995 Joyal and Moerdijk in their \emph{Algebraic Set Theory} proposed a new technique to guarantee the existence of internal models of $ZF$ and $IZF$ in a category. Before their work, all proposed categorical models of set theory were based on categories built starting from some standard models of set theory. Joyal and Moerdijk's approach has two main advantages:
\begin{enumerate}
\item it allows to prove the existence of an internal model of $IZF$ by simply checking some purely categorical properties of a class of arrows;
\item it works for (relatively) arbitrary categories.
\end{enumerate}
Their book gave rise to a large spread of works. It looked as algebraic set theory could give not only a useful tool to mathematics, but also something that could be interesting for foundational studies. 

There is maybe nothing more suggestive and more imprecise than the claim that \emph{Category theory provides a foundation for mathematics}. We don't mean that this is not true, but we must come to an agreement on what is the meaning of the words \emph{foundation for mathematics}. Kreisel observed that category theory provides a powerful tool to organize mathematics, and this is the way in which most mathematicians involved in category theory think about category theory as a foundation. There are also some attempts to use category theory to build a based-on-functions theory of sets on which it is possible to found the entire mathematical building; these are, for example, the axiomatic systems proposed by Lawvere \cite{LAW} and recalled by \cite{MCL}. In this paper we adopt the attitude to go in accordance with the first vision of \emph{category theory as foundation}. We think that category theory \emph{describes} mathematics in a very effective way. In the following pages this descriptive power will be exemplified by using category theory to clarify the role of classes in the practice of $ZF$ set theory. Category theory (in both its usual and internal version) will be used to give an account of the relation between mathematics and metamathematics.

\section{Classes and ZF}
Set theory is mainly $ZF$ (or better most of the set theorists studies $ZF$), that is the classical first-order theory with equality whose language has (countably many) individual variables, no functional symbols and a binary relational symbol $\in$, and whose specific axioms are the following
\begin{enumerate}
\item $\forall x \forall y (\forall t(t\in x \leftrightarrow t\in y)\rightarrow x=y)$;
\item for every formula $P$ with free variables $p_{1},...,p_{k},t$, 
$$\forall p_{1}...\forall p_{k}\forall x \exists y\forall t(t\in y\leftrightarrow (t\in x\wedge P));$$
\item $\forall x \forall y \exists z \forall t (t\in z\leftrightarrow (t=x\vee t=y))$;
\item $\forall x \exists z \forall t (t\in z\leftrightarrow \exists y(t\in y\wedge y\in x))$;
\item $\forall y \exists z \forall x(x\in z \leftrightarrow x\subseteq y)$;
\item for every formula $F$ with free variables $p_{1},..,p_{k},x,y$,
$$\forall p_{1}...\forall p_{k}\forall z (\forall x (x\in z\rightarrow \exists ! y F)\rightarrow \exists z' \forall x (x\in z \rightarrow \exists y (y\in z' \wedge F)));$$
\item $\exists x (0\in x\wedge \forall t(t\in x\rightarrow \sigma(t)\in x))$;
\item $\forall x(x\neq 0\rightarrow\exists z (z\in x \wedge \forall t (t\in x\rightarrow t\notin z)))$;
\end{enumerate}
where as usual $x\subseteq y$ is a shorthand for $\forall t(t\in x\rightarrow t\in y)$ and $\sigma(t)$ stays for $t\cup \left\{t\right\}$.

Set theorists work with sets, but they speak also about classes. However classes don't exist in $ZF$. The practice of set theorists consist in considering classes as shorthands or formal writings. If $P$ is a formula with a distinguished variable $x$, then a class is a formal writing $\left\{x|P(x)\right\}$ and the expression
$$t\in \left\{x|P(x)\right\}$$
is simply a shorthand for $P[t/x]$. For example set theorists write $t\in {\cal V}$, as a shorthand for $t\in \left\{x|x=x\right\}$, that is a shorthand for $t=t$. They also write $t\in {\cal ON}$ as a shorthand for $t\in\left\{x|ON(x)\right\}$, that is a shorthand for $ON(t)$ where $ON(x)$ is  the formula that expresses the fact that $x$ is an ordinal, that is
$$\forall s\forall s'\forall s''((s\in t\wedge s'\in t\wedge s''\in t\wedge s\in s'\wedge s'\in s'')\rightarrow s\in s'')\wedge$$
$$\wedge\forall s \forall s' ((s\in t\wedge s'\in t\wedge s\neq s')\rightarrow (s\in s'\vee s'\in s))\wedge \forall s (s\in t\rightarrow s\subseteq t).$$
What is clear, per se, is the fact that classes are not mathematical, but metamathematical objects. However these shorthands enjoy some of the properties that are enjoyed by sets. I can give a meaning to the notion of intersection of classes or to that of subclass.  This accidental fact makes some confusion breaking into! The metamathematical level and the mathematical one are confused giving rise to impressive, but heavily incorrect sentences like 
$$\textrm{"Sets are exactly those classes for which comprehension axiom is true".}(*)$$
It doesn't matter the fact that sentences like these can be used as some sort of \emph{convincing argument} to make students having an intuition of the notion of set, these sentences look more dangerous, than useful. 

The ease with which set theorists use classes (that are formal objects) as some quasi-sets is due to the dual fact that their syntactical structure (and in particular the use of connectives and quantifiers), as we have already seen, allows this for many operations and that they know that there exists a theory of classes (NBG) that is a conservative extension of ZFC. However classes are \emph{nothing more than syntactical objects}. 

In the next sections we will see how the use of category theory can help us to describe the relation between metamathematics and mathematics, and to distinguish between real sentences about sets and external (naive) sentences about sets. 
\section{The syntactic category of ZF}
We proceed now defining the syntactic category of $ZF$, that we will denote with $\mathbb{ZF}$. 
Its objects are formulas in context, i.e. formal writing as 
$$\left\{x_{1},..,x_{n}|P\right\},$$
where $x_{1},...,x_{n}$ is a (possibly empty) list of distinct variables and $P$ is a formula with all free variables among $x_{1},...,x_{n}$.
An arrow from a formula in context $\left\{\textbf{x}|P\right\}$ to another $\left\{\textbf{y}|Q\right\}$ is an equivalence class of formulas in context
$$[\left\{\textbf{x'},\textbf{y'}|F\right\}]_{\equiv},$$
where $\textbf{x'}$ is a list of variables that have the same lenght as $\textbf{x}$, and $\textbf{y'}$ is a list of variables that has the same lenght as  $\textbf{y}$, for which 
\begin{enumerate}
\item $F\vdash_{ZF} P[\textbf{x}'/\textbf{x}]\wedge Q[\textbf{y}'/\textbf{y}]$;
\item $F\wedge F[\textbf{y''}/\textbf{y'}]\vdash_{ZF} \textbf{y'}=\textbf{y''}$;
\item $P[\textbf{x}'/\textbf{x}]\vdash_{ZF}\exists \textbf{y'}F$. 
\end{enumerate}
The equivalence relation is given by
$$\left\{\textbf{x'},\textbf{y'}|F\right\}\equiv \left\{\textbf{x''},\textbf{y''}|F'\right\}\textrm{ iff }\vdash_{ZF}F\leftrightarrow F'[\textbf{x'}/\textbf{x''},\textbf{y'}/\textbf{y''}].$$
For composition it is sufficient to consider arrows $[\left\{\textbf{x'},\textbf{y'}|F\right\}]_{\equiv}$ and $[\left\{\textbf{y'},\textbf{z'}|F'\right\}]_{\equiv}$ with all distinct variables (this doesn't determine a loss of generality). In this case the composition is given by
$$[\left\{\textbf{x'},\textbf{z'}|\exists \textbf{y'}(F\wedge F')\right\}]_{\equiv}.$$
The category we obtain is regular and Boolean (see Johnstone \cite{JOH} for a proof).
In particular we recall the fact that 
$$\left\{\textbf{x}|P\right\}\times\left\{\textbf{y}|Q\right\}\cong\left\{\textbf{x},\textbf{y}|P\wedge Q\right\},$$
where we supposed (without loss of generality) that all involved variables are distinct.
Moreover terminal objects are given for example by
$$\left\{\;|\forall x(x=x)\right\},\;\left\{x|\forall t(t\notin x)\right\}.$$
We can also easily prove that $\mathbb{ZF}$ is extensive, thanks to the fact that $$\vdash_{ZF}0\neq 1.$$

Initial objects are given for example by
$$\left\{\;|\exists x(x\neq x)\right\},\;\left\{x|x\neq x\right\}.$$

$\mathbb{ZF}$ is also an exact category; the proof is based on \emph{Scott's trick} (see Rosolini \cite{ROS}, Maschio \cite{MAS}). Finally $\mathbb{ZF}$ has also a subobject classifier that is given by
$$[\left\{x,x'|x=0\wedge x'=1\right\}]:\left\{x|x=0\right\}\rightarrow \left\{x|x=0\vee x=1\right\}.$$

\section{Definable classes}
The category of definable classes of $ZF$, that we denote with $\mathbb{DCL}[ZF]$, is the full subcategory of $\mathbb{ZF}$ determined by all those objects that have list of variables of lenght $1$. It is clear that this is the category of classes, as they are intended in the practice of $ZF$.

The meaningful fact about $\mathbb{DCL}[ZF]$ is the fact that, although it is a subcategory of $\mathbb{ZF}$, it is provable to be equivalent to it. This result comes from the fact that in $ZF$ there is an available representation for ordered pairs:
$$x=<x_{1},x_{2}>\equiv^{def} x=\left\{\left\{x_{1}\right\},\left\{x_{1},x_{2}\right\}\right\}.$$ 

Every object of $\mathbb{ZF}$, $\left\{x_{1},..,x_{n}|P\right\}$, is isomorphic to $$\left\{x|\exists x_{1}...\exists x_{n}(x=<<...<x_{1},x_{2}>,...>,x_{n}>\wedge P)\right\},$$ where $x$ is a variable distinct from $x_{1},...,x_{n}$.
\section{The categorical side of class operations}
Before going on, we want to spend some words giving an example of the categorical interpretation of the practice of operations between classes. This is done to justify the adequacy of $\mathbb{DCL}[ZF]$ for the aim of describing the category of formal classes. For this purpose we focus on intersection. We consider two definable classes $\left\{x|P\right\}$ and $\left\{y|Q\right\}$ (we can suppose $x$ and $y$ to be distinct without loss of generality); we usually take their intersection to be $\left\{x,|P\wedge Q[x/y]\right\}$. However this operation can be expressed in purely categorical terms. In fact $\left\{x|P\wedge Q[x/y]\right\}$ is isomorphic in $\mathbb{DCL}[ZF]$ to the object we obtain by considering the following pullback.
\begin{diagram}
\left\{x|P\right\}\cap \left\{y|Q\right\}				&\rTo													&\left\{x|x=x\right\}\\
\dTo										&														&\dTo_{[\left\{x,x',y'|x=x'\wedge x=y'\right\}]_{\equiv}}\\
\left\{x,y|P\wedge Q\right\}						&\rTo_{[\left\{x,y,x',y'|P\wedge Q\wedge x=x'\wedge y=y'\right\}]_{\equiv}}				&\left\{x,y|x=x\wedge y=y\right\}\\
\end{diagram}
\section{Definable sets}
Let's now consider the full subcategory $\mathbb{DST}[ZF]$ of $\mathbb{DCL}[ZF]$ determined by the objects $\left\{x|P\right\}$ for which $$\vdash_{ZF}\exists z \forall x(x\in z\leftrightarrow P).$$
This is the category of sets we have in mind, when we think sentence $(*)$ is true.

We have the following result:
\begin{theorem}
$\mathbb{DST}[ZF]$ is a topos.
\end{theorem}
\begin{proof}
Finite limits are exactly finite limits in $\mathbb{ZF}$: a terminal object is given by $\left\{x|x=0\right\}$,  equalizers are exactly equalizers in $\mathbb{ZF}$, while the product of two definible sets $\left\{x|P\right\}$, $\left\{y|Q\right\}$ is given by
$$\left\{z|\exists x\exists y(z=<x,y>\wedge P\wedge Q)\right\}$$
with the obvious projections (we are assuming $x$ and $y$ to be distinct without loss of generality).
Subobject classifier is exactly the subobject classifier of $\mathbb{ZF}$, while exponentials $\left\{y|Q\right\}^{\left\{x|P\right\}}$ are given by
$$\left\{f|Fun(f)\wedge \forall s(s\in dom(f)\leftrightarrow P(s))\wedge \forall s'(s'\in ran(f)\rightarrow Q(s')) \right\},$$
with evaluation arrow given by
$$[\{F,y|\exists f \exists x(F=<f,x>\wedge Fun(f)\wedge \forall s(s\in dom(f)\leftrightarrow P(s))\wedge$$
$$\wedge  \forall s'(s'\in ran(f)\rightarrow Q(s')) \wedge <x,y>\in f)\}]_{\equiv}.$$
\end{proof}
\section{A system of axioms for algebraic set theory}
In this section we present a list of axioms for algebraic set theory proposed by Simpson \cite{SIM}.
Every axiomatization of algebraic set theory is based on a pair $(\mathbb{C},{\cal S})$ in which $\mathbb{C}$ is a category and ${\cal S}$ is a family of arrows of $\mathbb{C}$ (called family of \emph{small maps}). 
The following are the axioms:
\begin{enumerate}
\item $\mathbb{C}$ is a regular category;
\item The composition of two arrows in ${\cal S}$ is in ${\cal S}$;
\item Every mono is in ${\cal S}$;
\item (STABILITY) If $f\in{\cal S}$ and $f'$ is a pullback of $f$, then $f'\in{\cal S}$;
\item (REPRESENTABILITY) For every $X$, there exists an object ${\cal P}_{\cal S}(X)$ and an arrow $e_{X}:\in_{X}\rightarrow X\times {\cal P}_{\cal S}(X)$ with $\pi_{2}\circ e_{X}\in {\cal S}$ so that, for every $\psi:R\rightarrow X\times Z$ with $\pi_{2}\circ \psi\in {\cal S}$, there exists a unique arrow $\rho:Z\rightarrow {\cal P}_{\cal S}(X)$ that fits in a pullback as follows;
\begin{diagram}
R		&\rTo		&\in_{X}\\
\dTo^{\psi} &			&\dTo^{e_{X}}\\
X\times Z   &\rTo^{id_{X}\times \rho} &X\times  {\cal P}_{\cal S}(X)\\
\end{diagram} 
\item (POWERSET) $\sqsubseteq_{X}:\subseteq_{X}\rightarrow {\cal P}_{\cal S}(X)\times {\cal P}_{\cal S}(X)$ satisfies $\pi_{2}\circ \sqsubseteq_{X}\in {\cal S}$,

where $\sqsubseteq_{X}$ is determined by the following property:
\\

\emph{An arrow $f=<f_{1},f_{2}>:Z\rightarrow {\cal P}_{\cal S}(X)\times {\cal P}_{\cal S}(X)$ factorizes through $\sqsubseteq_{X}$ if and only if, considering the following couple of pullbacks},
\begin{diagram}
P		&\rTo	&\in_{X}		&\lTo			&Q\\
\dTo^{\pi}	&		&\dTo^{e_{X}}	&			&\dTo^{\pi'}\\
X\times Z	&\rTo^{id\times f_{1}}&X\times {\cal P}_{\cal S}(X)&\lTo^{id\times f_{2}}&X\times Z\\		
\end{diagram}
\emph{the arrow $\pi$ factorizes through} $\pi'$.
\end{enumerate}

We want also to recall some definitions:
\begin{definition} An object $U$ in a category $\mathbb{C}$ is \emph{universal}, if for every object $X$ in $\mathbb{C}$, there exists a mono $j:X\rightarrow U$.
\end{definition}
\begin{definition} An object $U$ in a regular category $\mathbb{C}$ with a class of small maps ${\cal S}$ is a \emph{universe} if there exists a mono $j:{\cal P}_{\cal S}(U)\rightarrow U$.
\end{definition}
\begin{definition} A $ZF$-\emph{algebra} for a regular category $\mathbb{C}$ with a class of small maps ${\cal S}$ is an internal sup-semilattice $(U,\subseteq)$ together with an arrow $\sigma:U\rightarrow U$, so that for every $\lambda: B\rightarrow U$ and for every $j:B\rightarrow A\in {\cal S}$, there exists $sup_{j}(\lambda):A\rightarrow U$ so that for any $j':B'\rightarrow A$ and $\lambda':B'\rightarrow U$, once we consider the following pullback,
\begin{diagram}
P		&\rTo^{\pi_{2}}	&B\\
\dTo^{\pi_{1}}	&		&\dTo^{j}\\
B'		&\rTo^{j'}		&A\\
\end{diagram}
we have that
$$sup_{j}(\lambda)\circ j'\subseteq \lambda'\textrm{ if and only if }\lambda\circ \pi_{2}\subseteq \lambda'\circ \pi_{1}.$$
\emph{Morphisms} of $ZF$-algebras are morphisms between internal sup-semilattices that preserve $sup_{j}(\lambda)$ along $j\in {\cal S}$ and commute with the arrows $\sigma$. An \emph{initial} $ZF$-algebra is an initial object in the category of $ZF$-algebras and morphisms between them.
\end{definition}

\section{Small maps}
We now want to define a class of small maps in $\mathbb{DCL}[ZF]$: the class ${\cal S}$. An arrow
\begin{diagram}
\left\{x|P\right\}	&\rTo^{[\left\{x,y|F\right\}]}		&\left\{y|Q\right\}\\
\end{diagram}
is in ${\cal S}$ if and only if 
$$\vdash_{ZF}\forall y \exists z \forall x(F(x,y)\leftrightarrow x\in z ).$$
The class ${\cal S}$ is a class of small maps in the sense of Simpson \cite{SIM}:
\begin{lemma} Every mono is in ${\cal S}$.
\end{lemma}
\begin{proof}
This is obtained using axiom 2.\;to obtain the existence of an empty set and axiom 3.\;to prove the existence of singletons.
\end{proof}
\begin{lemma} Compositions of arrows in ${\cal S}$ are in ${\cal S}$.
\end{lemma}
\begin{proof}
This is obtained using axioms 6.\;and 4. 
\end{proof}
\begin{lemma}
The stability axiom is true for ${\cal S}$.
\end{lemma}
\begin{proof}
This is obtained using axiom 6.
\end{proof}
\begin{lemma}
The representability axiom is satisfied by ${\cal S}$.
\end{lemma}
\begin{proof}
Fix a definible class $X=\left\{x|P\right\}$. The definible class ${\cal P}_{\cal S}(X)$ is given by 
$$\left\{y|\forall x(x\in y\rightarrow P)\right\},$$
while the definible class $\in_{X}$ is given by
$$\left\{z|\exists x \exists y(z=<x,y>\wedge \forall t(t\in y\rightarrow P(t))\wedge x\in y)\right\}$$
and the arrow $e_{X}:\in_{X}\rightarrow X\times {\cal P}_{\cal S}(X)$ is given by
$$[\left\{z,z'|\exists x \exists y(z=<x,y>\wedge \forall t(t\in y\rightarrow P(t))\wedge x\in y)\wedge z=z'\right\}]_{\equiv}.$$
Now if the following arrow 
\begin{diagram}
&\left\{z|\exists x \exists y (z=<x,y>\wedge R(x,y))\right\}	
&\rTo^{[\left\{z,z'|\exists x \exists y (z=<x,y>\wedge R(x,y))\wedge z=z'\right\}]_{\equiv}}	&\left\{x|P\right\}\times \left\{y|Q\right\}\\
\end{diagram}
is so that $R(x,y)\vdash_{ZF}P(x)\wedge Q(y)$ and $\vdash_{ZF}\forall y \exists y'\forall x (R(x,y)\leftrightarrow x\in y')$, that means that it represents (without loss of generality) a relation which has second component in ${\cal S}$, then its representing arrow from $\left\{y|Q\right\}$ to ${\cal P}_{\cal S}(X)$ is given by 
$$[\left\{y,y'|Q\wedge \forall x(R(x,y)\leftrightarrow x\in y')\right\}]_{\equiv}.$$ 
\end{proof}
\begin{lemma}
The powerset axiom is verified by ${\cal S}$.
\end{lemma}
\begin{proof}
The subset relation for $\left\{x|P\right\}$ is given by the following arrow 
$${[\left\{z,z'|\exists y \exists y'(z=<y,y'>\wedge y\subseteq y' \wedge \forall x (x\in y'\rightarrow P(x) ))\wedge z=z'\right\}]_{\equiv}}$$
from $\left\{z|\exists y \exists y'(z=<y,y'>\wedge y\subseteq y' \wedge \forall x (x\in y'\rightarrow P(x) ))\right\}$ to ${\cal P}_{\cal S}(X)\times {\cal P}_{\cal S}(X)$.\\
This relation is in ${\cal S}$ by virtue of axiom 5.
\end{proof}
We now have that small definible classes are exactly those classes $\left\{x|P\right\}$ for which the unique arrow to $1$, that is $[\left\{x,y|P\wedge y=0\right\}]$, is small. This means that 
$$\vdash_{ZF}\forall y\exists z \forall x(x\in z\leftrightarrow (P(x)\wedge y=0)).$$
Now we know that
$$\vdash_{ZF}\exists y (y=0)$$
and so the previous condition is equivalent to say that 
$$\vdash_{ZF}\exists z \forall x(x\in z \leftrightarrow P(x)).$$
This means that small definible classes are exactly definible sets.
We have also that $\left\{x|N(x)\right\}$ is a small definible class, where $N(x)$ is the formula saying that $x$ is a finite ordinal: this follows from axioms 7.\;and 2.

We finally have that $\left\{x|x=x\right\}$ is a universal definible class (and so also a universe), because, for every definible class
$\left\{x|P\right\}$, the arrow $[\left\{x,x'|P\wedge x=x'\right\}]_{\equiv}$ is a mono from it to $\left\{x|x=x\right\}$.

To conclude note that we have an explicit (and obvious) representation for an initial $ZF$-algebra: this is given by 
$$(\left\{x|x=x\right\},\sqsubseteq_{\left\{x|x=x\right\}},[\left\{x,z|z=\left\{x\right\}\right\}]_{\equiv}).$$
If $$[\left\{z,z'|F(z,z')\right\}]_{\equiv}:\left\{z|P\right\}\rightarrow \left\{z'|Q\right\}$$
is small in $\mathbb{DCL}[ZF]$, and $[\left\{z,x|\lambda(z,x)\right\}]_{\equiv}$ is an arrow from $\left\{z|P\right\}$ to $\left\{x|x=x\right\}$, then 
$$sup_{[\left\{z,z'|F(z,z')\right\}]}([\left\{z,x|\lambda(z,x)\right\}]_{\equiv})$$
is given by the arrow 
$$[\left\{z',x|Q\wedge \forall t(t\in x\leftrightarrow \exists z(F(z,z')\wedge \lambda(z,t)) )\right\}]_{\equiv}.$$

\section{Internal category theory}
The notion of internal category is the generalization of the notion of small category. Although we can define what is an internal category in an arbitrary category, we prefer to consider a category $\mathbb{C}$ with all finite limits. 
We have the following 
\begin{definition} An \emph{internal category} of $\mathbb{C}$ is a sestuple 
$$(C_{0},C_{1},\delta_{0},\delta_{1},ID,\Box)$$
in which $C_{0},C_{1}$ are objects of $\mathbb{C}$ and $\delta_{0},\delta_{1}:C_{1}\rightarrow C_{0}$, $ID:C_{0}\rightarrow C_{1}$, $$\Box:C_{1}\times_{\Box} C_{1}\rightarrow C_{1}$$ are arrows of $\mathbb{C}$, where the following is a pullback,
\begin{diagram}
C_{1}\times_{\Box} C_{1}	&\rTo^{p_{1}}	&C_{1}\\
\dTo^{p_{0}}					&		&\dTo^{\delta_{0}}\\
C_{1}				&\rTo^{\delta_{1}}	&C_{0}\\
\end{diagram}
that satisfy the following requests
\begin{enumerate}
\item $\delta_{1}\circ ID=\delta_{0}\circ ID=id_{C_{0}}$;
\item $\delta_{0}\circ \Box=\delta_{0}\circ p_{0}$ and $\delta_{1}\circ \Box=\delta_{1}\circ p_{1}$;
\item $\Box\circ \lceil ID\circ \delta_{0},id_{C_{1}}\rceil=\Box\circ \lceil id_{C_{1}},ID\circ \delta_{1}\rceil=id_{C_{1}}$;
\item $\Box\circ \lceil \Box\circ p_{0},p_{1}\rceil=\Box\circ \lceil p_{0},\Box\circ p_{1}\rceil \circ \lceil p_{0}\circ p_{0},\lceil p_{1}\circ p_{0},p_{1}\rceil\rceil:(C_{1}\times_{\Box}C_{1})\times_{\Box} C_{1}\rightarrow C_{0}$,
\end{enumerate}
where we denote with $\lceil f,f'\rceil$ the unique arrows that exist for the definitions of pullback and where $(C_{1}\times_{\Box}C_{1})\times_{\Box}C_{1}$ is the pullback of $\delta_{1}\circ \Box$ and $\delta_{0}$.
\end{definition}
Before going to the next section we show a way to externalize internal categories. This is done in a very natural way by means of global elements.
\begin{proposition}
If ${\cal C}=(C_{0},C_{1},\delta_{0},\delta_{1},ID,\Box)$ is an internal category of $\mathbb{C}$, then the following is a category:
$$\Gamma({\cal C}):=(Hom(1,C_{0}),Hom(1,C_{1}),\delta_{0}\circ (-),\delta_{1}\circ (-),ID\circ (-),\Box\circ \lceil (-)_{1},(-)_{2}\rceil)$$
\end{proposition}
\begin{proof}
Every point of the definition of category follows immediately because of the relative point in the definition of internal category. The proof is straightforward. 
\end{proof}

\section{The real category of sets: ${\cal SET}$}
We will now define an internal category of $\mathbb{ZF}$ (or equivalently of $\mathbb{DCL}(ZF)$), called ${\cal SET}$. This category is given by the following assignments:
\begin{enumerate}
\item ${\cal SET}_{0}:=\left\{x|x=x\right\}$;
\item ${\cal SET}_{1}:=\left\{F|\exists f \exists z(F=<f,z>\wedge Fun(f)\wedge ran(f)\subseteq z)\right\}$;
\item $\delta_{0}:=[\left\{F,x|\exists f \exists z(F=<f,z>\wedge Fun(f)\wedge ran(f)\subseteq z\wedge dom(f)=x)\right\}]_{\equiv}$;
\item $\delta_{1}:=[\left\{F,z|\exists f (F=<f,z>\wedge Fun(f)\wedge ran(f)\subseteq z)\right\}]_{\equiv}$;
\item $ID:=[\left\{x,F|\exists f(F=<f,x>\wedge \forall t(t\in f \leftrightarrow \exists s(s\in x \wedge t=<s,s>)))\right\}]_{\equiv}$
\item $$\Box:=[\{J,G|\exists f \exists f' \exists z \exists f''(J=<f,<f',z>>\wedge Fun(f)\wedge Fun(f')$$
$$\wedge ran(f)\subseteq dom(f')\wedge ran(f')\subseteq z\wedge G=<f'',z>\wedge $$
$$\wedge\forall t(t\in f'' \leftrightarrow (\exists s \exists s' \exists s''(<s,s'>\in f \wedge <s',s''>\in f' \wedge t=<s,s''>) ) ))\}]_{\equiv},$$ once we easily realized that the object of composable arrows is given by\\
$$\{J|\exists f \exists f' \exists z (J=<f,<f',z>>\wedge Fun(f)\wedge$$
$$\wedge Fun(f')\wedge ran(f)\subseteq dom(f')\wedge ran(f')\subseteq z)\}.$$
\end{enumerate}
We have that
\begin{theorem}
${\cal SET}$ is an internal category.
\end{theorem}
\begin{proof}
This is a straightforward proof.
\end{proof}
We have also that this is an internal topos, as every construction for a topos can be done in $\mathbb{ZF}$, as one can see (this is given by a well written formal proof of the fact that sets and functions form a topos).
\section{Global elements: names for sets}
Now we have an internal category ${\cal SET}$; we want to study the category $\Gamma({\cal SET})$. As follows directly from the definition we have that the objects of $\Gamma({\cal SET})$ are the equivalence classes $[\left\{x|P(x)\right\}]_{\equiv}$ of definable classes so that $\vdash_{ZF}\exists ! x P(x)$.
Arrows of $\Gamma({\cal SET})$ are classes of equivalence $[\left\{f|P(f)\right\}]_{\equiv}$ so that $\vdash_{ZF}\exists ! f P(f)$, and 
$$P(f)\vdash_{ZF}\exists f'\exists z(f=<f',z>\wedge Fun(f')\wedge ran(f')\subseteq z).$$
We have the following 
\begin{theorem}
$\mathbb{DST}(ZF)$ and $\Gamma({\cal SET})$ are equivalent.
\end{theorem}
\begin{proof}
Consider the following functors:
\begin{enumerate}
\item $\textbf{P}:\mathbb{DST}(ZF)\rightarrow \Gamma({\cal SET})$ is given by\\
$\textbf{P}(\left\{x|P(x)\right\}):=[\left\{z|\forall x(x\in z\leftrightarrow P(x))\right\}]_{\equiv}$\\
$\textbf{P}([\left\{x,y|F(x,y)\right\}]_{\equiv}):=$\\
$:=[\{f'|\exists f \exists z(f'=<f,z>\wedge \forall t(t\in f \leftrightarrow \exists x \exists y(t=<x,y>\wedge F(x,y)))\wedge$\\
$\wedge \forall y(y\in z\leftrightarrow Q(y)))\}]_{\equiv}$\\
$\textbf{P}(\left\{y|Q(y)\right\}):=[\left\{z|\forall y(y\in z\leftrightarrow Q(y))\right\}]_{\equiv}$.\\
\item $\textbf{P'}:\Gamma({\cal SET})\rightarrow \mathbb{DST}[ZF]$ is given by\\
$\textbf{P'}([\left\{z|P(z)\right\}]_{\equiv}):=\left\{x_{0}|\exists z(P(z)\wedge x_{0}\in z)\right\}$\\
$\textbf{P'}([\left\{f'|Q(f')\right\}]_{\equiv}):=[\left\{x,x'|\exists f'(Q(f')\wedge \exists z\exists f(f'=<f,z>\wedge <x,x'>\in f))\right\}]_{\equiv}$\\
$\textbf{P'}([\left\{z|P'(z)\right\}]_{\equiv}):=\left\{x_{0}|\exists z(P'(z)\wedge x_{0}\in z)\right\}$,\\
where $x_{0}$ is a fixed variable (we can think of it as the first variable if we consider variables of $ZF$ to be presented in a countable list).
\end{enumerate}
It is immediate to see that $\textbf{P}\circ \textbf{P'}$ is the identity functor for $\Gamma({\cal SET})$, while there is an natural isomorphism from the identity functor of $\mathbb{DCL}[ZF]$ to $\textbf{P'}\circ \textbf{P}$, that is given by the arrows
$$[\left\{x,x'|P(x)\wedge x=x'\right\}]_{\equiv}:\left\{x|P(x)\right\}\rightarrow \left\{x_{0}|\exists z(x_{0}\in z \wedge \forall x(x\in z \leftrightarrow P(x)))\right\}$$
\end{proof}
\section{Final remarks}
In our attempt to clarify the relation between (formal) classes and sets, between metamathematics and mathematics, by means of a unique mathematical structure, we started by introducing the syntactical category $\mathbb{ZF}$ (that we proved to be equivalent to the category of definable classes of $ZF$).
This category corresponds to the metamathematical level: its objects are classes as are usually introduced in the set theorists' practice. Moreover this category has a full subcategory of some importance: the category of definable sets, that is the category whose objects are those definable classes $\left\{x|P\right\}$ for which 
$$\vdash_{ZF}\exists z \forall x(x\in z\leftrightarrow P).$$ 
This is the naive category of sets. Obviously this is \emph{not} the real category of sets. The \emph{real} category of sets is ${\cal SET}$. But this is not a category: it is an \emph{internal category} in $\mathbb{ZF}$.
The relation between metamathematics and mathematics is exactly the relation between categories and internal categories. Mathematical concepts are represented through internal categories, external (or metamathematical) concepts are expressed at the categorical level. 
The most interesting result shown in the previous sections is that showing the equivalence of the two more natural ways to give an external account of the notion of set. We proved that the category of definable sets is equivalent to the category obtained by global sections on ${\cal SET}$: \emph{classes that satisfy comprehension axiom are exactly those to which I can give a name}.

\bibliographystyle{plain}

\begin{thebibliography}{1}

\bibitem{}
S.Awodey, C.Butz, A.Simpson, T.Streicher.
\newblock Relating topos theory and set theory via categories of classes.
\newblock {\em Bulletin of Symbolic Logic, 13(3): 340-358}, 2007.

\bibitem{}
Y.Bar-Hillel, A.A.Fraenkel, A.Levy.
\newblock  Foundations of Set Theory.
\newblock {\em North Holland}, 1973.

\bibitem{}
E. Casari.
\newblock Questioni di filosofia della matematica.
\newblock{\em Feltrinelli}, 1964.
  
  
\bibitem{ENG}
H.R.E.Engler.
\newblock On the Problem of Foundations of Category Theory.
\newblock{\em Dialectica, 3(1)}, 1969.

\bibitem{FEF}
S.Feferman.
\newblock Categorial foundations and foundations of category theory.
\newblock{\em Logic, foundations of mathematics and computability theory: 149-169}, 1977.


\bibitem{}
T.Jech.
\newblock Set theory. The third millennium edition.
\newblock{\em Springer-Verlag}, 2003.

\bibitem{JOH}
P.Johnstone.
\newblock Sketches of an elephant: a topos theory compendium, vol.2.
\newblock{\em Oxford University Press}, 2002.


\bibitem{}
A.Joyal, I.Moerdijk,
\newblock Algebraic Set Theory.
\newblock{\em Cambridge Unviersity Press}, 1995.

\bibitem{}
G.Kreisel.
\newblock Observations of popular discussions on foundations.
\newblock{\em Axiomatic Set Theory, American Mathematical Society: 183-190}, 1971

\bibitem{}
K.Kunen
\newblock Set Theory: An introduction to Independence Proofs.
\newblock{\em North Holland}, 1983.

\bibitem{LAW}
F.W.Lawvere.
\newblock An elementary theoy of the category of sets (long version) with commentary.
\newblock{\em Reprints in Theory and Applications of Categories, 11: 1-35}, 2005

\bibitem{LOL}
G.Lolli.
\newblock Categorie, universi e principi di riflessione.
\newblock{\em Bollati Boringhieri},1977

\bibitem{}
S.MacLane.
\newblock Category theory for the working mathematicians.
\newblock{\em Springer-Verlag},1972


\bibitem{}
M.Makkai, G.E.Reyes.
\newblock First order categorical logic.
\newblock{\em Springer}, 1977.

\bibitem{MAS}
S.Maschio.
\newblock Aspects of internal set theory.
\newblock{\em PhD thesis}, 2012.

\bibitem{MCL}
C.McLarty.
\newblock Exploring categorical structuralism.
\newblock{\em  Philosophia Mathematica (3) Vol.12: 37-53}, 2004.

\bibitem{}
G.Rosolini
\newblock Una teoria di classi e insiemi.
\newblock{\em Lecture notes}, 2008

\bibitem{ROS}
G.Rosolini.
\newblock La categoria delle classi definibili in IZF \`e un modello di AST.
\newblock{\em XXIV incontro di logica AILA}, 2011.

\bibitem{SIM}
A.Simpson.
\newblock Elementary axioms for category of classes.
\newblock{\em Logic in Computer Science, pp. 77-85}, 1999.

\end{thebibliography}

\end{document}